\documentclass[journal,twoside,web]{ieeecolor}

\usepackage{amsthm,amsmath,amssymb,amsfonts}
\usepackage{generic}
\usepackage{cite}
\usepackage{algorithm}
\usepackage{algpseudocode}
\usepackage{graphicx}
\usepackage{textcomp}
\usepackage{color}
 \usepackage{etex}
\usepackage{lipsum}
\usepackage{amsfonts}
\usepackage{graphicx}
\usepackage{epstopdf} 
\usepackage{enumerate} 
 \usepackage{bbm,enumerate,multirow,framed,blkarray}

\usepackage{verbatim,tcolorbox}
  \usepackage{tikz}
\usetikzlibrary{shapes,arrows}
\usetikzlibrary{positioning}  
\usepackage{pstricks-add} 
 \usepackage{wrapfig}
 \usepackage{pgfplots}
  
\usepackage{color}
\usepackage{savesym}

\newtheorem{theorem}{Theorem}[section]

\newtheorem{definition}[theorem]{Definition} 
\newtheorem{remark}[theorem]{Remark}

\newtheorem{example}[theorem]{Example}

	\algnewcommand\And{\textbf{and }}
	\algnewcommand\Or{\textbf{or }}

\newcommand{\xnot}{\x_0}
\newcommand{\wnot}{\w_0}

\newcommand{\znot}{\mathbf{z}_0}

\newcommand{\btheta}{\pmb{\theta}}

\newcommand{\coord}{\mathbf{e}} 
 
\newcommand{\capw}{\mathbf{W}}
\newcommand{\caps}{\mathbf{S}}

\newcommand{\real}{\mathbb R}
\newcommand{\posint}{\mathbb{N}}

\newcommand{\zz}{\mathbf{z}}
\newcommand{\w}{\mathbf{w}}

\newcommand{\uu}{\mathbf{u}}
\newcommand{\vv}{\mathbf{v}}

\newcommand{\f}{\mathbf{f}}

\newcommand{\F}{\mathbf{F}}

\newcommand{\g}{\mathbf{g}}
\newcommand{\dd}{\mathbf{d}}
\newcommand{\h}{\mathbf{h}}
\newcommand{\x}{\mathbf{x}}

\newcommand{\y}{\mathbf{y}}
\newcommand{\M}{\mathbf{M}}
\newcommand{\I}{\mathbf{I}}

\newcommand{\II}{\mathbf{I}}

\newcommand{\GG}{\mathbf{G}}

\newcommand{\m}{\mathbf{m}}

\newcommand{\rr}{\mathbf{r}}
\newcommand{\capx}{\mathbf{X}}

\newcommand{\barcapD}{\Bar{\mathbf{x}}}
\newcommand{\barcapA}{\Bar{\mathbf{w}}}

\newcommand{\capu}{\mathbf{U}}
\newcommand{\capv}{\mathbf{V}}

\newcommand{\capy}{\mathbf{Y}}

\newcommand{\zero}{\mathbf{0}}

\newcommand{\Jf}{\mathbf{Jf}}
\newcommand{\J}{\mathbf{J}}

\algnewcommand{\algorithmicgoto}{\textbf{go to}}%
\algnewcommand{\Goto}[1]{\algorithmicgoto~\ref{#1}}%
 
\newenvironment{bsmallmatrix}{\left[\begin{smallmatrix}}{\end{smallmatrix}\right]}

\usepackage{hyperref}
\hypersetup{hidelinks=true}
\usepackage{comment}
\usepackage{graphicx}
\usepackage{subcaption}
\usepackage{epsfig} 
\usepackage{cite}
\usepackage{bm} 
\usepackage{caption}
\usepackage{wrapfig}
\usepackage{cutwin}
\usepackage{epstopdf}
\usepackage{lcsys}
\newcommand{\hesham}[1]{\textcolor{black}{#1}} 
\newcommand{\ps}[1]{\textcolor{black}{#1}} 

\begin{document}

\def\BibTeX{{\rm B\kern-.05em{\sc i\kern-.025em b}\kern-.08em
    T\kern-.1667em\lower.7ex\hbox{E}\kern-.125emX}}
\markboth{\journalname, VOL. XX, NO. XX, XXXX 2017}
{Author \MakeLowercase{\textit{et al.}}: Preparation of Papers for IEEE Control Systems Letters (August 2022)}

\title{\LARGE \bf
Observability and State Estimation for Smooth and Nonsmooth Differential Algebraic Equation Systems
}

\author{Hesham Abdelfattah$^{1}$, Sameh A. Eisa$^{2}$, and Peter Stechlinski$^{3}$
\thanks{$^{1}$ Hesham Abdelfattah is with the Department of Aerospace
Engineering and Engineering Mechanics and the Department of
Mathematical Sciences, University of Cincinnati, Cincinnati, OH, USA
        {\tt\small abdelfhm@mail.uc.edu}}%
\thanks{$^{2}$ Sameh A. Eisa is with the Department of Aerospace Engineering and Engineering Mechanics, 
        University of Cincinnati, OH, USA
        {\tt\small eisash@ucmail.uc.edu}}%
\thanks{$^{3}$ Peter Stechlinski is with the Department of Mathematics and Statistics, University of Maine, ME, USA
        {\tt\small peter.stechlinski@maine.edu}}%
\thanks{This material is based upon work supported by the National Science Foundation under Award No. 2318772 and 2318773.}
}

\maketitle
\thispagestyle{empty}

\begin{abstract}
In this work, 
we extend the sensitivity-based rank condition (SERC) test for local observability to another class of systems, namely smooth and nonsmooth differential-algebraic equation (DAE) systems of index-1. The newly introduced test for DAEs, which we call the lexicographic SERC (L-SERC) observability test, utilizes the theory of lexicographic differentiation to compute sensitivity information. 
Moreover, the newly introduced L-SERC observability test 
can judge which states are observable and which are not. 
Additionally,  we introduce a novel sensitivity-based extended Kalman filter (S-EKF) algorithm for state estimation, applicable to both smooth and nonsmooth DAE systems.  Finally, 
 we apply the newly developed S-EKF to estimate the states of a wind turbine power system model.
\end{abstract}


\section{Introduction}
\IEEEPARstart{O}{bservability} analysis aims to determine if we are able to identify the states of the system via access to the measurements of some output function. In mathematical terms, observability of an input-output dynamical system means that we can uniquely identify the initial conditions of the states via measurements of the output function. 
In the literature, a common method to study local observability of ODE systems is the observability
rank condition test, which is based on successive Lie
derivative computations involving the output function \cite{hermann1977nonlinear}. There are other methods in the literature that study local observability, such as the sensitivity rank condition test (SERC) \cite{stigter2018efficient}.   Recently, the authors in \cite{van2022sensitivity} showed that the SERC test for observability (and its closely related property, identifiability) is a more computationally efficient method when compared to Lie derivative-based tests, especially for large-scale systems. For ODE systems, establishing (or assuming) system observability has led to a matured literature of designing observers for state estimation, such as but not limited to extended Kalman filtering (EKF) \cite[Chapter 8]{beard2012small} \hesham{with convergence analysis of EKF discussed in \cite{ljung1979asymptotic}}; recently, geometric-based EKF was also introduced for control-affine ODEs \cite{pokhrel2023gradient}. It is important to highlight that all the aforementioned results are for smooth ODE input-output systems. Recently, the authors of this paper introduced a novel nonsmooth analog of the SERC test, the lexicographic SERC (L-SERC) test, which enabled the observability test for nonsmooth ODE systems \cite{stechlinski2025identifiability}. For piecewise ODE systems (i.e., ones which transitions between different modes), a discontinuous EKF (DEKF) algorithm was introduced in \cite{chatzis2017discontinuous} and its convergence was studied --- their approach operates on subintervals within each mode but does not theoretically treat the nonsmoothness  in a direct way.

In this paper, we are concerned with another important class of input-output systems, namely, the smooth and nonsmooth DAEs of index-1.
Such systems are widely present in different engineering applications and physical models. For example, different applications where smooth DAEs are utilized are presented in
\cite{brenan1995numerical,kunkel2006differential}, and we see that nonsmooth DAE systems \cite{2015_Stechlinski_Barton_continuousdependence,2016_Stechlinski_Barton_LD_DAEs, Stechlinski_Patrascu_Barton_DAEF} are used in modeling process systems \cite{2017_NonsmoothDAE_CACE}, multibody mechanical systems  \cite{pang2007strongly}, and power systems \cite{eisa2021sensitivity,abdelfattah2025new}.  
In the literature, multiple works have been aimed at studying observability of smooth DAEs, such as the study of observability of linear DAEs (see \hesham{\cite{campbell1991duality,campbell1991observability,hou2002causal}}), and of nonlinear DAEs (see \cite{terrell1997observability,montanari2024identifiability}). There have also been studies for designing observers for smooth DAEs, such as the
the application of  EKF methods to DAEs in \cite{becerra2001applying} and \cite{mandela2010recursive}. Nevertheless, to the best of our knowledge, sensitivity-based methods, such as SERC, have not been introduced to study observability of DAE systems (smooth or nonsmooth).    
 

\textbf{Contributions}:
    In this paper, we develop a sensitivity-based test (the L-SERC test) for assessing local observability for both smooth and nonsmooth DAE input-output systems of index-1. 
    Moreover, the newly
introduced L-SERC test algorithm is useful as it judge which states are observable and which are not. Additionally,  
we introduce a novel sensitivity-based extended Kalman filter (S-EKF) algorithm for state estimation, applicable to both smooth and nonsmooth DAE systems.  Finally, 
 we apply the newly developed S-EKF to estimate the states of a wind turbine power system model.

\section{Preliminaries}\label{sec:prelim}
Throughout this paper, we will use the lexicographic directional derivative (LD-derivative) \cite{Khan2015_automaticdiff} as a tool for obtaining derivative information of a given 
lexicographically smooth (L-smooth) \cite{Nesterov2005} function, 
\hesham{where a locally Lipschitz function  $\f:\real^n \rightarrow \real^m$ is said to be L-smooth at $\xnot$ \cite{Nesterov2005} if, for any $k \in \posint$ and directions matrix $\M=[\m_1  \quad \cdots \quad \m_k] \in \real^{n \times k}$, the homogenization sequence 
\begin{equation}\label{eq:Lsmoothfunctions}
\begin{aligned}
\f^{(0)} &: \real^n \to \real^m: \dd \mapsto \f'(\xnot;\dd),\\
\f^{(j)}&: \real^n \to \real^m: \dd \mapsto [\f^{(j-1)}_{\xnot,\bm{\mathrm{M}}}]'(\m_j;\dd), \; j=1,\ldots,k,
\end{aligned}
\end{equation}
exists, where $\f'(\xnot;\dd)=\lim_{\alpha \to 0^+} [\f(\xnot+\alpha \dd)-\f(x)]/\alpha$.  \ps{(See \cite{ref:OMS_generalizedderivatives} for an example calculating the homogenization sequence.)}
All piecewise  differentiable ($PC^1$) \cite{Scholtes} (e.g., min, max, abs-value) and  convex  (e.g., p-norms) functions, and their compositions, are L-smooth.
The LD-derivatives are defined as follows using the homogenization sequence \cite{Khan2015_automaticdiff}:}\small
\color{red}\begin{equation}\label{eq.LDD}
\f'(\xnot;\M):=[\f^{(0)}(\m_1)  \quad \cdots \quad \f^{(k-1)}(\m_k)]\in\real^{m \times k},
\end{equation}
\normalsize\color{black} and are computationally relevant objects as they satisfy sharp calculus rules, hence providing a practical, theoretically rigorous, numerical toolkit \cite{ref:OMS_generalizedderivatives}. In addition, the LD-derivatives theory provides closed-form formulas for widely used nonsmooth functions 
such as the absolute-value function and the max/min functions. 
For example, given \ps{$\M\in\real^{2 \times k}$,}
\ps{\small{\begin{align}
    \min{'}(x_0,y_0;\bm{\mathrm{M}})
                &= {\rm \bf slmin}([x_0 \quad \text{row}_1(\M)],[y_0 \quad \text{row}_2(\M)])  \label{eq.LDderivative.max}
\end{align}}}\normalsize
\ps{where {\small ${\rm \bf slmin}=\text{row}_1(\M)$ }\normalsize if {\small $\mathrm{fsign}(x_0-y_0,m_{11}-m_{21},\ldots,m_{1k}-m_{2k}) \leq 0$ }\normalsize and {\small ${\rm \bf slmin}=\text{row}_1(\M)$ }\normalsize otherwise, with $\mathrm{fsign}$  returning the sign of the first nonzero element (or zero if the input is zero). }\normalsize
\ps{(See \cite{ref:OMS_generalizedderivatives} for other closed-form expressions, and the LD-derivative calculus rules.)}

The lexicographic derivative (L-derivative) \cite{Nesterov2005}, $\J_{\rm L}\f(\xnot;\M)\ps{\in\real^{m \times n}}$,  
is a computationally relevant Jacobian-like object, i.e., it can be used in nonsmooth numerical methods \cite{bagirov2020numerical_BGK+20} as it can substitute an  element  of Clarke's generalized derivative  \cite{Clarkebook}:  
\small\begin{align*}
\partial \f(\xnot)
:=\text{conv}
\{\F : & \; \exists \x_{j} \to \xnot,  \text{ s.t. }  \x_{j} \in S_{\f},   \J \f(\x_{j})  \to \F\},
\end{align*}
\normalsize
 where $S_{\f}$ is the full measure subset of $\text{dom}(\f)$ on which $\f$ is differentiable. With L-derivatives being the desirable object to compute, and LD-derivatives enjoying a strong   toolkit for calculation, the picture is completed by noting that an L-derivative can be obtained from an LD-derivative when the matrix $\M$ has full row rank (and is hence right invertible):
\small\begin{equation}\label{eq:LDderivative:vs:Lderivative}
\ps{\f'(\xnot;\M)=\J_{\rm L}\f(\xnot;\M)\M}.
\end{equation} 
\normalsize
Note that if the function is at least $C^1$, then it is true that  $\f'(\x_0;\M)=\Jf(\x_0) \M$ (hence, the L-derivative equals the Jacobian in this case). 
Lastly, for later use in discussing nonsmooth DAEs, we recall that the Clarke generalized derivative projection \cite{Clarkebook} of $\g$ with respect to $\w$ at $(\xnot,\wnot)$ is  
\small\begin{align*}
\pi_{\w}\partial \g(\xnot,\wnot):= \left\{\GG_{\w} : \exists  [\GG_{\x} \quad \GG_{\w}]\in\partial \g(\xnot,\wnot)\right\}.
\end{align*}
\normalsize
That is, the class of L-smooth functions are locally Lipschitz and directionally differentiable to arbitrary order, which includes all $C^1$, $PC^1$ \cite{Scholtes}, and convex functions, and compositions thereof.  

\section{Sensitivity-Based Observability Method}\label{sec:LSERC test}
\subsection{L-SERC Test for Nonsmooth DAE Observability}\label{subsection:LSERC_test}
In this subsection, we extend the L-SERC observability test from nonsmooth ODEs \cite{stechlinski2025identifiability} to nonsmooth DAEs (and thus SERC observability to smooth DAEs).
Consider an input-output DAE system 
in semi-explicit form: 
\small\begin{subequations}\label{eq:1}
\begin{align}
 \dot{\x}(t) &= \f(\x(t),\w(t),\uu(t)), \quad  \x(t_0) = \x_0,\label{eq:1a}\\
 \zero&=\g(\x(t),\w(t),\vv(t)), \label{eq:1b}\\
 \y(t) &= \h(\x(t),\w(t),\uu(t),\vv(t)),\label{eq:1c} 
 \end{align}
  \end{subequations}
\normalsize with $n_x$ differential states $\x$ and $n_w$ algebraic states $\w$. Here there are $n_y$ outputs of the system $\y$, and we assume (throughout the article) that the admissible control inputs satisfy $\uu \in L^1([t_0,t_f],D_u)$ (i.e., Lebesgue-integrable) $\vv  \in LS^0([t_0,t_f],D_v)$ (i.e., L-smooth). We also assume that the  right-hand side (RHS)  functions satisfy $\f \in LS^0(D_x \times D_w \times D_u,\real^{n_x})$ and $\g \in LS^0(D_x \times D_w \times D_v,\real^{n_w})$, and the output  function satisfies $\h \in LS^0(D_x\times D_w \times D_u \times D_v,\real^{n_y})$, with open and connected sets $D_x \subseteq \real^{n_x}$, $D_u \subseteq \real^{n_u}$, $D_v \subseteq \real^{n_v}$, $\Theta \subseteq \real^{n_p}$, and $\x_0 \in D_x$ is the initial state. In what follows, we let $n_z=n_x+n_w$, $D_z=D_x\times D_w$, $\zz=(\x,\w)$ and $\zz_0=(\xnot,\wnot)$. Before proceeding, we formalize  (local) partial observability (in the spirit of partial identifiability \cite{stechlinski2025identifiability}) and regularity (which implies generalized differentiation index one  (see \cite{2015_Stechlinski_Barton_continuousdependence,Stechlinski_Patrascu_Barton_DAEF})  of a nonsmooth DAE, given reference initial conditions $\zz_0^*=(\x_0^*,\w_0^*)$ and reference  control inputs $\uu^*,\vv^*$.  

\begin{definition}\label{defn:identifiability.partial} 
The system in \eqref{eq:1} is  \textbf{locally partially  observable} (or simply \textbf{partially  observable}) at $\zz^*_0$ if  there exist a neighborhood $N \subseteq D_z$ of $\znot^*$ and a connected set $V \subseteq D_z$ containing $\zz^*_0$ such that  for any $\tilde{\zz}_0, \zz^{\dagger}_0\in N \cap V$, we have  that, for all $t\in[t_0,t_f]$,  $\y(t;\uu^*,\vv^*,\tilde{\zz}_0)=\y(t;\uu^*,\vv^*,\zz^{\dagger}_0)$ if and only if $\tilde{\zz}_0=\zz^{\dagger}_0$.
\end{definition}

\begin{definition}\label{defn:regular} A solution 
$\zz^*(t):=\zz(t;\uu^*,\vv^*)=(\x^*(t),\w^*(t))$ of \eqref{eq:1} on $[t_0,t_f]$ through $\{(\x_0^*,\w^*_0,\uu^*,\vv^*)\}$ is \textbf{Clarke-regular} (or simply \textbf{regular}) if every matrix in 
$
\pi_{\w}\partial \g(\x^*(t),\w^*(t),\hesham{\vv^*(t)})$
is nonsingular for all $t\in[t_0,t_f]$. 
\end{definition}
 
\hesham{Note that in the case that $\g$ is $C^1$, then it holds that $
\pi_{\w}\partial \g(\x^*(t),\w^*(t),\vv^*(t))=\frac{\partial \g}{\partial \w}(\x^*(t),\w^*(t),\vv^*(t))$ and so Definition \ref{defn:regular} recovers classical differentiation index one.}
We establish an observability L-SERC test for partial observability, based on the lexicographic sensitivity (L-sensitivity) functions of the output. 

\begin{theorem}\label{thm:nonsmoothSERC.observability}
Suppose that $\zz^*=(\x^*,\w^*)$ is a regular solution of \eqref{eq:1} on $[t_0,t_f]$ through $\{(\x_0^*,\w^*_0,\uu^*,\vv^*)\}$ and, for some $\{t_0,t_1,\ldots,t_N\} \subset [t_0,t_f]$ and $\dd\in\real^{n_x}$, it holds that ${\rm rank}({\pmb \Upsilon}_{\dd})=n_x$, where
\small{\begin{equation}\label{LSERC_ID}
{\pmb \Upsilon}_{\dd}:= 
(\caps_{\y}^{\rm L}(t_0),
\caps_{\y}^{\rm L}(t_1),
\dots,
\caps_{\y}^{\rm L}(t_N))
\end{equation}}
\normalsize
and the L-sensitivity output functions are
 \small{\begin{equation}\label{eq.Lsensitivities}
\caps_{\y}^{\rm L}(t)=\h'(\x^*,\w^*,\uu^*,\vv^*;(\capx^*,\capw^*,\zero,\zero)) \begin{bsmallmatrix}
    \zero_{1\times n_x}\\ \I_{n_x}
\end{bsmallmatrix},
\end{equation}}
\normalsize
and $(\capx^*,\capw^*)$ uniquely solve the following  on $[t_0,t_f]$: 
\small{\begin{subequations}\label{eq:sens.nonsmooth}
\begin{align}
&\dot{\capx}
=\f'(\x^*,\w^*,\uu^*;(\capx,\capw,\zero), \quad \capx(t_0)
=[\dd \quad \II_{n_x}],\label{eq:sens.nonsmootha}\\
&\zero
=\g'(\x^*,\w^*,\vv^*;(\capx,\capw,\zero)).\label{eq:sens.nonsmoothb}
\end{align}
\end{subequations}}
\normalsize
Then \eqref{eq:1} is partially observable. 
\end{theorem} 

\begin{proof} 
The proof follows from combining Proposition 1 and Theorem 1 in \cite{abdelfattah2024parameter}; by choosing $\f_0(\btheta)=\btheta$, and replacing the notation $\btheta$  and $\btheta^*$ by  $\x_0$ and $\x_0^*$ (i.e., $n_p=n_x$), respectively,  \cite[Proposition 1]{abdelfattah2024parameter} yields that the nonsmooth forward sensitivity system associated with \eqref{eq:1} is given as follows (see Equation (4) in \cite{abdelfattah2024parameter}):
\small\begin{subequations}\label{eq:sens.nonsmooth.inproof}
\begin{align}
&\dot{\capx}
=\f'(\x^*,\w^*,\uu^*;(\capx,\capw,\zero)), \quad \capx(t_0)
=\M,\label{eq:sens.nonsmootha.inproof}\\
&\zero
=\g'(\x^*,\w^*,\vv^*;(\capx,\capw,\zero)),\label{eq:sens.nonsmoothb.inproof}\\
&\capy
=\h'(\x^*,\w^*,\uu^*,\vv^*;(\capx,\capw,\zero,\zero)),\label{eq:sens.nonsmoothc.inproof} 
\end{align}
\end{subequations}
\normalsize
for some directions matrix $\M \in \real^{n_x \times n_k}$ (since $n_p=n_x$) and integer $n_k \in \posint$, where the $(t)$ arguments have been omitted. Then, choosing $\M=[\dd \quad \II_{n_x}]$ (hence, $n_k=n_x+1$) and applying \cite[Theorem 1]{abdelfattah2024parameter} yields that \eqref{eq:1} is partially observable if ${\rm rank}({\pmb \Upsilon}_{\dd})=n_x$ in \eqref{LSERC_ID} holds, where
\small{\begin{equation}\label{eq.yLsens} 
\caps_{\y}^{\rm L}(t)=\capy^*(t) [\dd \quad \II_{n_x}]^{-1} =\capy^*(t) \begin{bsmallmatrix}
    \zero_{1\times n_x}\\ \I_{n_x}
\end{bsmallmatrix},
\end{equation}}
\normalsize
with $\capy^*$ obtained from \eqref{eq:sens.nonsmoothc.inproof} (and thus  \eqref{eq.Lsensitivities}). $\quad\square$
\end{proof}
\normalsize
The L-SERC test depends on the values of the L-sensitivity output function at different time samples, $\{t_0,t_1,\ldots,t_N\}$. In particular, 
for a given direction $\dd \in \real^{n_x}$, we construct the L-SERC matrix ${\pmb \Upsilon}_{\dd}={\pmb \Upsilon}_{\dd}(\znot^*)\in \real^{(N+1)n_y \times n_x}$ by concatenating the values of the L-sensitivity output function at different time samples to get \eqref{LSERC_ID}, as in \cite{stechlinski2025identifiability,abdelfattah2024parameter}. In line with these works, we note that Theorem \ref{thm:nonsmoothSERC.observability} establishes a test for observability based on whether the system is ``L-SERC observable'' or ``L-SERC non-observable'':
\begin{definition}\label{defn:identifiability.partial.indirection}
The system \eqref{eq:1} is   \textbf{L-SERC observable} at $\znot^*$ in the direction $\dd \in \real^{n_x}$ if $\text{rank}({\pmb \Upsilon}_{\dd})=n_x$, and \textbf{L-SERC non-observable} if $\text{rank}({\pmb \Upsilon}_{\dd})<n_x$.
\end{definition}


\begin{remark}
The L-SERC matrix provides local observability conclusions in the direction $\dd$. The significance of the direction here is that for nonsmooth systems/output functions, studying observability depends on the piecewise nature of the system.
For example, for a nonsmooth output function $y=\textnormal{max}(x,0)$, the test will conclude observability if we use $d=1$ (since this ``probing'' leads to $y=x>0$ and thus observability), while if we use $d=-1$ we get non-observability (since this ``probing'' leads to $y=0$).
\end{remark}

\begin{remark}\label{remark.gendiffone}
Despite $\dd$ being a vector in $n_x$-dimensional Euclidean space, it still provides probing information  concerning the observability of $\zz_0=(\x_0,\w_0)$ in $\real^{n_z}=\real^{n_x+n_w}$. This is a consequence of  the differential state initial conditions $\x_0$ being considered as the parameters, rather than $(\x_0,\w_0)$ together, and the solution being regular: once $(\x_0^*,\w_0^*)$ is chosen, only $n_x$ degrees of freedom are available as $\w$ is constrained by $\g(\ps{\x,\w},\vv)=\zero$ as the DAE system being generalized differentiation index one guarantees a nonsmooth implicit function $\w=\rr(\x)$ semi-locally to the reference solution $\x^*(t)$ corresponding to $(\x_0^*,\w_0^*)$ (see \cite{2015_Stechlinski_Barton_continuousdependence,2016_Stechlinski_Barton_LD_DAEs}). The rough idea then is that if $\x_0$ can be determined from $\y$, then $\w_0$ can be determined from $\x_0$.
\end{remark}



\begin{remark}\label{remark:smoothDAE}
\normalsize In case we have a smooth DAE system, i.e., the RHS functions in  \eqref{eq:1} are $C^1$, and $\M=\II_{n_x}$ is chosen, \eqref{eq:sens.nonsmooth.inproof} simplifies as follows (without needing to check for smoothness/nonsmoothness): 
\small\begin{align}\label{eq:sens.smooth}
\dot{\caps}_{\x}(t)
&= \tfrac{\partial \f}{\partial \x} \caps_{\x}(t)+ \tfrac{\partial \f}{\partial \w} \caps_{\w}(t),\quad\caps_{\x}(t_0)
=\I_{n_x},\notag\\
\zero 
&= \tfrac{\partial \g}{\partial \x} \caps_{\x}(t)+ \tfrac{\partial \g}{\partial \w} \caps_{\w}(t),\\
\caps_{\y}(t) &= \tfrac{\partial \h}{\partial \x} \caps_{\x}(t)+ \tfrac{\partial \h}{\partial \w} \caps_{\w}(t)\notag,
\end{align}
\normalsize
 which is uniquely solved on $[t_0,t_f]$ by the classical  sensitivities \ps{$\caps_{\x}=\tfrac{\partial \x}{\partial \xnot}, \caps_{\w}=\tfrac{\partial \w}{\partial \xnot}, \caps_{\y}=\tfrac{\partial \y}{\partial \xnot}$}.
The SERC observability test for smooth DAEs gives a yes/no answer based on whether ${\rm rank}({\pmb \Upsilon})=n_x$ or ${\rm rank}({\pmb \Upsilon})<n_x$, respectively, where 
\small{\begin{align}\label{eq:SERC_Matrix}
   {\pmb \Upsilon}:=
    (\caps_{\y}(t_0),
    \caps_{\y}(t_1),
    \dots,
    \caps_{\y}(t_N)).
\end{align}}
\end{remark}

\subsection{Determining Observable and Non-Observable States}
The observability L-SERC test presented above is used not only to determine the observability of the smooth/nonsmooth DAE system, but also to determine which of the system's states are non-observable. That is, it determines the states that are not possible to be identified from the system's output, hence causing the non-observability of the system.
To make this more precise, we introduce the following definition.

\begin{definition}\label{defn:param.identifiability.partial} 
The  differential states and algebraic states in 
$\bm{\chi}_{\rm lno}\subseteq\{\x_1,\ldots,\x_{n_x}\}$ and $\bm{\alpha}_{\rm lno}\subseteq\{\w_1,\ldots,\w_{n_w}\}$
are \textbf{locally partially non-observable} (or simply \textbf{partially non-observable}) at $\znot^*\in D_z$ if there exists a neighborhood $N \subseteq D_z$ of $\znot^*$ and a connected set $V \subseteq D_z$ containing $\znot^*$ such that the states in $\bm{\chi}_{\rm lno} \cup \bm{\alpha}_{\rm lno}$ are locally non-observable in $N\cap V$. 
\end{definition}

As a part of our L-SERC observability algorithm, we determine which states are locally partially non-observable near $\znot^*$ by analyzing the singular vectors produced from the singular value decomposition (SVD) of the L-SERC output matrix.
We first determine which differential states are non-observable: given the SVD  $\textstyle {\pmb \Upsilon}_{\dd}=\capu \Sigma \capv^{\rm T}$, 
with  left singular vectors matrix  $\capu\in\real^{(N+1)n_y\times (N+1)n_y}$,  singular values matrix $\Sigma\in\real^{(N+1)n_y \times n_x}$, and  right singular vectors matrix $\capv \in\real^{n_x \times n_x}$, if  ${\rm rank}({\pmb \Upsilon}_{\dd})=n_r<n_x$, then the right nullspace ${\rm Null}({\pmb \Upsilon}_{\dd}) \neq \emptyset$. Similar to the procedure outlined in \cite[Section III-B]{abdelfattah2024parameter}, we can identify the zero singular values $\sigma_{n_r+1}=\cdots=\sigma_{n_x}=0$ and it follows that the $n_x-n_r$ columns in $\capv$ associated with these zero singular values 
form a basis for $\text{Null}({\pmb \Upsilon}_{\dd})$. Accordingly, we construct a matrix whose columns are those $n_x-n_r$ singular vectors, 
\small\begin{align}\label{eq:Pmatrix}
    \capv_r:=\begin{bmatrix}
    \vv_{(n_r+1)}&\ldots&\vv_{(n_x)}
\end{bmatrix},
\end{align}
\normalsize from which we can  {identify} the non-observable differential states by placing $\capv_r$ in row echelon form, $\textnormal{rref}(\capv_r^{\rm{T}})$, and identifying pivot columns; if row $j$ of $\textnormal{rref}(\capv_r^{\rm{T}})$ contains the pivot from pivot column $i$, then $\x_i$ is  partially non-observable.

With non-observable differential states $\bm{\chi}_{\rm lno}$ identified, we can shift to  {identifying} non-observable algebraic states $\bm{\alpha}_{\rm lno}$ (thanks to generalized differentiation index one --- see Remark \ref{remark.gendiffone}): 
For each $\w_i$, we construct the matrix 
\small\begin{equation}\label{eq.Psi}
\bm{\Psi}_{\w_i}:=(\caps_{\w_i}^{\rm L} (t_0),\dots,\caps_{\w_i}^{\rm L}(t_N)), 
\end{equation}
\normalsize
where $\caps_{\w_i}^{\rm L}(t_k)=[(\caps_{\w}^{\rm L}(t_k))]_{i,\mathcal{J}}$ (i.e., $i^{\rm th}$ row with columns indexed by $\mathcal{J}$) and $\mathcal{J}=\{j:\x_j \in \bm{\chi}_{\rm lno}\}$. Then $\bm{\Psi}_{\w_i}$  characterizes to the sensitivity of the algebraic $\w_i$ with respect to the non-observable differential states and so if $\text{rank}(\bm{\Psi}_{\w_i})\neq 0$, then $\w_i\in \bm{\alpha}_{\rm lno}$ is non-observable.

With the ideas above in mind, we present the DAE L-SERC observability test in Algorithm \ref{algo.LSERC}, which requires the reference initial conditions, the reference control inputs, and a set of probing directions ($D=\{\dd_i\}$).
The algorithm returns the set of locally partially non-observable differential states, $\bm{\chi}_{\rm lno}$, the set of locally partially non-observable differential states, $\bm{\chi}_{\rm lo}=\{\x_1,\ldots,\x_{n_x}\}\setminus \bm{\chi}_{\rm lno}$, the set of locally partially non-observable algebraic states $\bm{\alpha}_{\rm lno}$,  and the set of locally partially observable algebraic states $\bm{\alpha}_{\rm lo}=\{\w_1,\ldots,\w_{n_w}\}\setminus \bm{\alpha}_{\rm lno}$.
Implementation of Algorithm \ref{algo.LSERC} can be found in \cite{githubdae_obs}.

\begin{algorithm}
\caption{L-SERC Observability Algorithm.} 
\label{algo.LSERC}
 \begin{algorithmic}[1] 
\renewcommand{\algorithmicrequire}{\textbf{Input:}}
\Require $\zz_0^*,\uu^*,\vv^*,D=\{\dd_i\}, \{t_0,t_1,\ldots,t_N\}$
\State Set $\bm{\chi}_{\rm lno}\gets\emptyset,\bm{\alpha}_{\rm lno}\gets\emptyset$
\label{algo.jump} 
\For{$i=1,2,\ldots,|D|$}
    \Statex \vspace{-.3cm}\hrulefill\vspace{-.1cm}
    \State Compute   ${\pmb \Upsilon}_{\dd_i}=\capu\Sigma\capv^{\rm{T}}$,  $n_r = \text{rank}(\Sigma)$  
    \If {$n_r<n_x$}
    \State $\capv_r =[\vv_{(k): \sigma_k=0}]$,  $\text{rref}(\capv_r^{\rm{T}})=[\tilde{\vv}_{1}\; \cdots \;\tilde{\vv}_{n_x}]$
    \State Set $\bm{\chi}_{\rm lno} \gets \bm{\chi}_{\rm lno} \cup \{\x_j\}:\tilde{\vv}_j$ is pivot column 
    
   \EndIf 
\EndFor
    \Statex \vspace{-.3cm}\hrulefill\vspace{-.1cm}
\For{$i=1,2,\ldots,n_w$}
    \Statex \vspace{-.3cm}\hrulefill\vspace{-.1cm}
    \State Compute $\bm{\Psi}_{\w_i}$
    \If {$\text{rank}(\bm{\Psi}_{\w_i})\neq 0$}
    \State Set $\bm{\alpha}_{\rm lno} \gets \bm{\alpha}_{\rm lno} \cup \{\w_i\}$
   \EndIf 
\EndFor
    \Statex \vspace{-.3cm}\hrulefill\vspace{-.1cm}\\
\Return ${\pmb \Upsilon}_{\dd_i}$ for all $\dd_i \in D$\\
\Return $\bm{\chi}_{\rm lno},\bm{\alpha}_{\rm lno},\bm{\chi}_{\rm lo}=\{\x_j\}\setminus \bm{\chi}_{\rm lno},\bm{\alpha}_{\rm lo}=\{\x_j\}\setminus \bm{\chi}_{\rm lno}$ 
\end{algorithmic}
\end{algorithm} 

\section{Sensitivity-based EKF Method}
An important application of the observability property is the design of dynamic observers, with EKF and its variations being an important technique. In this section, we present a variation of EKF, the sensitivity-based EKF (S-EKF), that utilizes the sensitivity information from \eqref{eq:sens.nonsmooth.inproof} to provide accurate estimations of the system's observable, or even partially-observable,  states. Consider the DAE system given by \eqref{eq:1}, but with process noise and measurement noise added to the system, i.e., the DAE system is given by 
\small\begin{subequations}\label{eq:EKF_DAE}
\begin{align}
 \dot{\x}(t) &= \f(\x(t),\w(t),\uu(t))+\bm{\Omega}(t), \quad  \x(t_0) = \x_0,\label{eq:EKF_DAE_1a}\\
 \zero&=\g(\x(t),\w(t),\vv(t)), \label{eq:EKF_DAE_1b}\\
 \y(t) &= \h(\x(t),\w(t),\uu(t),\vv(t))+\bm{\nu}(t),\label{eq:EKF_DAE_1c} 
 \end{align}
  \end{subequations}
\normalsize where $\bm{\Omega}(t)\in\real^{n_x}$ is the process noise, i.e.,  zero mean, uncorrelated, independent continuous random
variables with covariance $\bm{Q}\in\real^{n_x\times n_x}$ and $\bm{\nu}(t)\in\real^{n_y}$ is a measurement noise, i.e., discrete uncorrelated independent random variables with zero mean and covariance $\bm{R}\in\real^{n_y\times n_y}$. The goal of the S-EKF is that, for some  model that can be predicted by a DAE system (smooth or nonsmooth), we use the ``approximate'' DAE model in \eqref{eq:1} to give accurate estimation of the states of the ``true'' DAE model in \eqref{eq:EKF_DAE}, with the process noise representing the uncertainty in the modeling and the measurement noise representing sensor limitations. 
  
  In the proposed S-EKF, the sensitivity functions are utilized in the calculation of the covariance matrix. In particular, we solve the following ``augmented'' nonsmooth forward sensitivity system (cf.\ the non-augmented version in \eqref{eq:sens.nonsmooth.inproof})  with the covariance matrix, $\bm{P}(t)$, added as an extra state:
\small\begin{subequations}\label{eq:EKF_cov_DAE}
\begin{align}
\dot{\capx}
&=\f'(\x^*,\w^*,\uu^*;(\capx,\capw,\zero)), \quad \capx(t_0)
=\M,\label{eq:EKF_cov_DAE_1a}\\
\zero
&=\g'(\x^*,\w^*,\vv^*;(\capx,\capw,\zero)),\label{eq:EKF_cov_DAE_1b}\\
\capy
&=\h'(\x^*,\w^*,\uu^*,\vv^*;(\capx,\capw,\zero,\zero)),\label{eq:EKF_cov_DAE_1c}\\
\dot{\bm{P}} &= \begin{bsmallmatrix} \capx \\ \capw \end{bsmallmatrix} \bm{P} + \bm{P} \begin{bsmallmatrix} \capx \\ \capw \end{bsmallmatrix}^{\bm{T}} +\bm{Q}, \quad \bm{P}(t_0)=\bm{P}_0\label{eq:EKF_cov_DAE_1d},
\end{align}
\end{subequations}
\normalsize where $(\x^*,\w^*)$ are the solutions of \eqref{eq:1} corresponding to $\x_0^*$, which admits the unique solution $(\capx^*,\capy^*)$ on  $[t_0,t_f]$. Here $\M$ acts as a ``probing matrix'' for sensitivity information and should be chosen to be full rank, while the choice of $\bm{P}_0$ should be made based on the confidence in the accuracy of the initial conditions $\x^*_0$ (see \cite[Chapter 5]{brown1992introduction} for more details).
 
\begin{remark}\label{rmk:smooth_S-EKF}
If the DAE system is smooth and $\M=\II_{n_x}$, \ps{\eqref{eq:EKF_cov_DAE_1a}-\eqref{eq:EKF_cov_DAE_1c} simplifies
to \eqref{eq:sens.smooth}.}
\end{remark} 
\normalsize We split up the time horizon $[t_0,t_f]$ into $n_s$ subintervals $[t_{k-1},t_k]$,  with $t_{n_s}=t_f$ and with $t_k$ being the times of receiving measurements; we define  the filter differential states  at each $t_k$ as $\barcapD_{t_k}:=\x^*(t_k)$, and the filter algebraic states  $\barcapA_{t_k}:=\w^*(t_k)$.
For the measurement equation, we let
\small{\begin{align}\label{eq:EKF_measurementEqn}
    \bar{\y}_{t_k}
    &:=\h(\barcapD_{t_k},\barcapA_{t_k},\uu^*(t_k),\vv^*(t_k)).
\end{align}}
\normalsize The L-derivative of the measurement is calculated as \small\begin{equation}\label{eq:EKF_measurementEqn.lderivative}
\bm{C}_{t_k} := \caps_{\y}^{\rm L}(t_k)=\capy^*(t_k)\M^{-1}.
\end{equation}
\normalsize Finally, we note that we perform the L-SERC test on each subinterval $[t_{k-1},t_{k}]$ in the S-EKF algorithm to determine which states are observable and which states are not. The states that are observable are updated from the measurements $\y_{m}(t_k)$ (the subscript $m$ denotes measurements) using the observer gain $\bm{L}_{t_k}$, and the states that are non-observable assume their predicted values with no update from the measurements. That is, the states that are non-observable are calculated from the approximate DAE model in \eqref{eq:1}. We note that on any subinterval $[t_{k-1},t_{k}]$, the vector $\bm{L}_{t_k}$ should have zero entries corresponding to the non-observable states. If all states are non-observable, we set $\bm{L}_{t_k}=\bm{0}$.

\begin{algorithm}
\caption{S-EKF State Estimation Algorithm.} 
\label{algo.EKF}
 \begin{algorithmic}[1] 
\renewcommand{\algorithmicrequire}{\textbf{Input:}}
\Require $\x_0,\w_0,\uu^*,\vv^*,\capx_0=\M,\capw_0,\bm{P}_0,\{t_0,t_1,\ldots,t_{n_s}\}$
\State Set $\barcapD\gets \{\x_0\}$, $\barcapA\gets \{\w_0\}$, $N \gets \lceil {\frac{n_x}{n_y}} \rceil-1$ 
\For{$k=1,\ldots,n_s$}
    \Statex \vspace{-.3cm}\quad \;\hrulefill\vspace{-.1cm}
\State   
  \textbf{Prediction step}
  \State Calculate $(\x(t_k),\w(t_k))$ by solving \eqref{eq:1} on $[t_{k-1},t_k]$ 
   \Statex  \quad \;    with ICs $(\x(t_{k-1}),\w(t_{k-1}))$
\State   
  Set $\barcapD_{t_k}\gets\x(t_k),\barcapA_{t_k}\gets\w(t_k),\Bar{\zz}_{t_k}\gets(\barcapD_{t_k},\barcapA_{t_k})$ 
    \Statex \vspace{-.3cm}\quad \; \hrulefill\vspace{-.1cm}
     \State \textbf{Measurement step}
    \State Obtain measurement value $\y_{m}(t_{k})$ 
    \State Calculate  $\bar{\y}_{t_k} \gets \h(\barcapD_{t_k},\barcapA_{t_k},\uu^*(t_k),\vv^*(t_k))$
     \State Calculate $(\capy(t_k),\bm{P}(t_k))$ by solving  \eqref{eq:EKF_cov_DAE} on $[t_{k-1},t_k]$
  \Statex  \quad \;    with ICs $(\capx(t_{k-1}),\capw(t_{k-1}),\bm{P}(t_{k-1}))$
  \State Set $\bm{C}_{t_k}\gets\caps^{\rm L}_{\y}(t_{k})=\capy(t_k)\M^{-1}$
\State Set $\bm{L}_{t_k} \gets \bm{P}(t_{k})\bm{C}_{t_k}^{\bm{T}}(\bm{R}+\bm{C}_{t_k}\bm{P}(t_{k})\bm{C}_{t_k}^{\bm{T}})^{-1}$
    \State Set $\zz_0^*\gets (\x(t_{k-1}),\w(t_{k-1}))$
     \State Set $\Delta t_k \gets (t_k-t_{k-1})/N$
     \State Set $\{T_j\} \gets \{t_{k-1},t_{k-1}+\Delta t_k,\ldots,t_{k-1}+N \Delta t_k\}$
    \State
   Run Algo  \ref{algo.LSERC}  
      with $\zz_0^*,\uu^*,\vv^*$, 
 $\{\pm\coord_i\}$, $\{T_j\}$  
\For{$\x_i \in \bm{\chi}_{\rm lno},\w_j \in \bm{\alpha}_{\rm lno}$}
\State Set $[\bm{L}_{t_k}]_{i,j} \gets 0$
\EndFor 
\State Set estimate $\Bar{\zz}_{t_k} \gets \Bar{\zz}_{t_k} + \bm{L}_{t_k} (\y_{m}(t_{k})-\bar{\y}_{t_k})$
\State Set $\barcapD \gets \barcapD \cup \{\barcapD_{t_k}\}$, $\barcapA \gets \barcapA \cup \{\barcapA_{t_k}\}$ 
    \Statex \vspace{-.3cm}\quad \;\;\hrulefill\vspace{-.1cm}
  \State Set new ICs $\x(t_{k})\gets\barcapD_{t_k}$
    \State Calculate consistent $\w(t_k)$ by solving \eqref{eq:1b}  at $t=t_k$ 
  \State Set new ICs $\capx(t_k) \gets \M$
  \State Set new ICs $\bm{P}(t_{k}) \gets (\I - \bm{L}_{t_k}\bm{C}_{t_k})\bm{P}(t_{k})$
  \State Calculate consistent  $\capw(t_{k})$ by solving \eqref{eq:EKF_cov_DAE_1b} at $t=t_k$
\EndFor
\State\Return $\barcapD,\barcapA$\label{algo.param.red.return} 
\end{algorithmic}
\end{algorithm}

\section{Applications \hesham{and Comparison with Literature}}
\ps{In this section, we present two examples to demonstrate the effectiveness of our proposed L-SERC test and S-EKF algorithm. 
The MATLAB code for the examples and the tests can be found in \cite{githubdae_obs}.} 
\begin{example}(\hesham{Smooth} Wind Turbine Power System) 
    \ps{Consider a wind turbine power system with constant wind speed (i.e., with constant active power and time-changing reactive power), represented by  the following DAEs \cite[Subsection 3.2]{tsourakis2009effect}}: 
    \small\begin{align}\label{ex:WTPS DAE}
        &\dot{V}_{ref} = K_{Q_i}(Q_{cmd}-Q)+\Omega_{\hesham{1}}(t)\hesham{=f_1+\Omega_1(t),} \notag\\
        &\dot{E}''_q = K_{V_i}(V_{ref}-V)\hesham{+\Omega_2 (t)=f_2+\Omega_2 (t)},\\
        &0=V^4-[2(PR+QX)+E^2]V^2+(R^2+X^2)(P^2+Q^2)\hesham{=g}, \notag
    \end{align}
\normalsize where $\x=(V_{ref},E''_q)$ are the differential states, $w=V$ is the algebraic state, $\btheta=(K_{Q_i},K_{V_i},R,X,E)$ are the system parameters, and \hesham{$\bm{\Omega}(t)=(\Omega_1(t),\Omega_2(t))$ is the process noise}. Here $V_{ref}$ represents the reference terminal voltage, $E''_q$ represents the equivalent voltage controlling the reactive current injection, and $V$ is the terminal voltage connecting the wind turbine and the grid. The system parameters $K_{Q_i}=0.1$, $K_{V_i}=40$ are integral control gains, $R=0.02$ is the net resistance, $X=0.02987$ is the net reactance, and $E=1.0164$ is the infinite bus voltage. We also note that $X_{eq}=0.8$ is the equivalent Norton reactance, $Q_{cmd}=0.6484$ is the constant reactive power command, $P=1$ is the rated power, and $Q=V(E''_q-V)/X_{eq}$ is the injected reactive power.

\hesham{We consider the time interval $[t_0,t_f]=[0,1]$ and the initial conditions $\znot=(V_{ref,0},E''_{q,0},V_0) = (0.5,0.75,1.021)$ and $\bm{P}_0=4\I$. The output function is the smooth function $y=E''_qV+\nu(t)$, where we assume a Gaussian white measurement noise $\nu(t)$. This output shows the effectiveness of the S-EKF in decoupling the states in the output, providing accurate estimations for each state individually. The left panel in Figure \ref{fig:EKF_combined} shows the outcome of the S-EKF Algorithm \ref{algo.EKF}, which produces an estimation (black curve) that accurately tracks the true solution (blue curve) for the differential states.
}

\hesham{Observability tests are performed on the system in \eqref{ex:WTPS DAE}, as well as the output, with no noise, i.e., $\Omega(t)=\nu(t)=0$. For the observability L-SERC test (Algorithm \ref{algo.LSERC}),   we consider ten time samples, i.e., $N=10$ uniformly between $t_0=0$ and $t_f=1$, and we construct the matrix $\Upsilon$ in \eqref{eq:SERC_Matrix}. We see that $\text{rank}(\Upsilon)=2=n_x$, i.e., the system in \eqref{ex:WTPS DAE} is locally observable at $\znot$. The local observability of a nonlinear ODE system can be determined by the rank of the observability matrix that is constructed from the gradients of successive Lie derivatives of the output function \cite{hermann1977nonlinear}.}
\hesham{The regularity of the solution of the index-1 DAE system in \eqref{ex:WTPS DAE} guarantees the existence of a semi-local map $V=V(\x)$ in a neighborhood around the solution trajectory of interest  \cite{2015_Stechlinski_Barton_continuousdependence}  
(the derivation of the formula for $V(\x)$ can be seen in \cite{eisa2019modeling}), so the system in \eqref{ex:WTPS DAE} can be transformed to a nonlinear ODE with two differential states in $\x$, as mentioned in Remark \ref{remark.gendiffone}, and in \cite[Subsection 2.2]{montanari2024identifiability}. The gradient of $y(\x)=E''_{q}V(\x)$ is 
\small\begin{align*}
    \nabla y(\x)=\left[ \frac{\partial y}{\partial V_{ref}}(\x), \; \frac{\partial y}{\partial E''_{q}}(\x) \right]^{\rm T}=\left[ 0 \quad V(\x)+E''_{q}\frac{\partial V}{\partial E''_{eq}}(\x) \right]^{\rm{T}}\end{align*}}
\normalsize \hesham{where} 
\hesham{\small{\begin{align*}    
\frac{\partial V}{\partial E''_{q}} &= - \frac{\partial g / \partial E''_{q}}{\partial g / \partial V}
=  \frac{2E''_{q}V^2}{4V^3 - 2V[2(PR+QX)+x_2^2]}.
\end{align*}}}
\normalsize
 \hesham{
 The Lie derivative  $L_{\f} y(\x) = \nabla y(\x) \cdot \f(\x)$  along  $\f=(f_1,f_2)$ is $L_{\f} y(\x)= \left(V+E''_{q}\frac{\partial V}{\partial E''_{eq}}\right) K_{V_i}(V_{ref} - V)=(V+\kappa)f_2.$
}
\normalsize 
\hesham{The observability matrix $\mathbf{O}(\x)$ for this 2-state system is 
\small\begin{align*}
\mathbf{O}(\x) &= 
\frac{\partial}{\partial \x}\begin{bmatrix}
    y(\x) \\
    L_{\f} y(\x)
\end{bmatrix}
=
\begin{bmatrix}
    0 & V(\x)+\kappa(\x) \\
    K_{V_i}V(\x)+\kappa(\x) & \frac{\partial L_{\f} y}{\partial E''_{q}}(\x)
\end{bmatrix}.
\end{align*}
\normalsize 
Noting that $\text{rank}(\mathbf{O}(\xnot))= 2$ (i.e., full rank), the system in \eqref{ex:WTPS DAE} is locally observable at $\xnot$, matching the results of our L-SERC test. Now we perform the local observability test proposed in \cite{terrell1997observability} for the DAE system in \eqref{ex:WTPS DAE}. We note that this test can be directly applied to a DAE system. The observability matrix $\mathbf{O}(\zz)=[\frac{\partial \f }{\partial \zz},\frac{\partial y}{\partial \zz}]^{\rm T}$ for the system in \eqref{ex:WTPS DAE} is}
\small\hesham{\begin{align*}
\mathbf{O}(\zz) &=\begin{bmatrix}
    0& -K_{Q_i} \frac{V}{X_{eq}}&-K_{Q_i} \frac{E''_q - 2V}{X_{eq}} \\
     K_{V_i}&0&- K_{V_i}\\
     0&- \left[ 2X \frac{V}{X_{eq}} + 2E''_q \right]V^2&\frac{\partial g}{\partial V}\\
    0&V&E''_{q}
\end{bmatrix},
\end{align*}
\normalsize where $\partial g/\partial V=4V^3 - 2V[2(PR+QX)+E''^2_q] - 2XV^2\left({E''_q - 2V}\right)/X_{eq}.$ 
    \normalsize Note that $\text{rank}(\mathbf{O}(\znot))=3$ (i.e., full rank), the system is locally observable at $\znot$, also matching the results of L-SERC.}
\end{example}    
\normalsize
\hesham{\begin{example}(Nonsmooth Wind Turbine Power System) 
    Now we consider the system in \eqref{ex:WTPS DAE} with the same initial conditions, but with the nonsmooth output function $y=\textnormal{min}(V+\nu(t),0.98)$. This nonsmooth output represents the case when the sensor can only measure voltage values below a certain threshold (taken to be $0.98$ here). 
    The right panel in Figure \ref{fig:EKF_combined} shows the outcome of Algorithm \ref{algo.EKF} applied on $[t_0,t_f]=[0,1]$, with the same initial conditions as the first case: because the voltage starts greater than $0.98$, the system is non-observable for all states and the S-EKF algorithm estimates the states using their predicted values.
    After $t=0.057$, the value of the voltage becomes less than $0.98$, hence the output function switches to $y=V$, 
    and all the states become observable. This is reflected by the state estimates tracking their true values. The output is not strongly correlated with the first state, $V_{ref}$, 
    and so the tracking accuracy of $V_{ref}$ is not high. The nonobservability behavior is indeed reflected in the outcome of the L-SERC test: we see that all the rows of the L-SERC matrix $\Upsilon_{\coord_1}$ in \eqref{LSERC_ID} are the zero vectors for $t\in[0,0.057)$, i.e., $\caps_{\y}^{\rm L}(t)=\zero,\; t\in[0,0.057)$. Starting from $t=0.057$, the rows of $\Upsilon_{\coord_1}$ have nonzero entries and are linearly independent, implying local observability of the system at $\znot$ for $t\in[0.057,1]$. The advantage of our proposed L-SERC algorithm and S-EKF algorithms is made clear in this case: it is not possible in this example (with the nonsmooth output function) to perform the other smooth observability tests. 
\end{example}}

\begin{figure}[h!]
  \centering
  \begin{minipage}[t]{0.48\columnwidth}
    \centering
    \includegraphics[width=\linewidth]{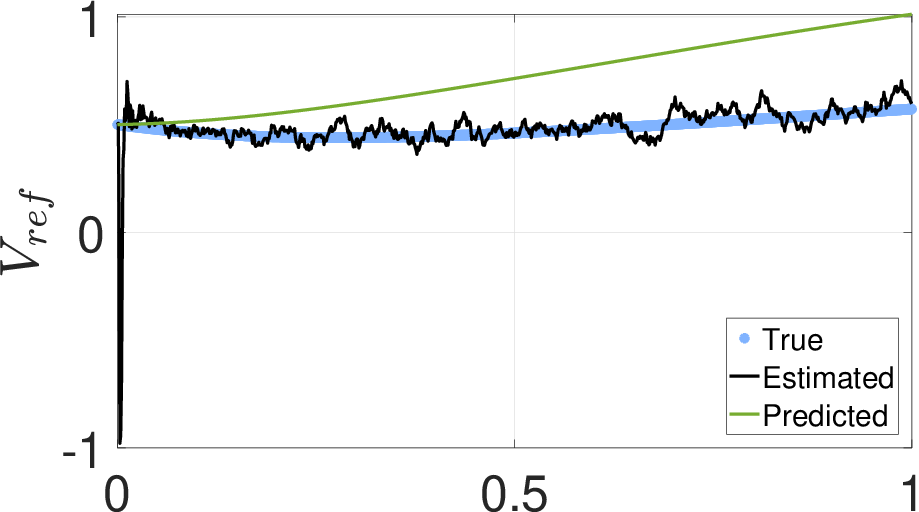}\\[1mm]
    \includegraphics[width=\linewidth]{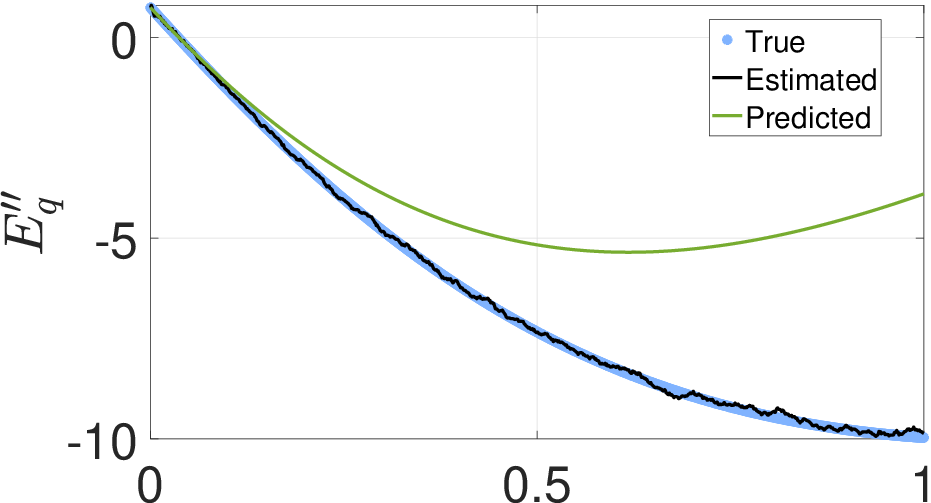}\\[1mm]
    \includegraphics[width=\linewidth]{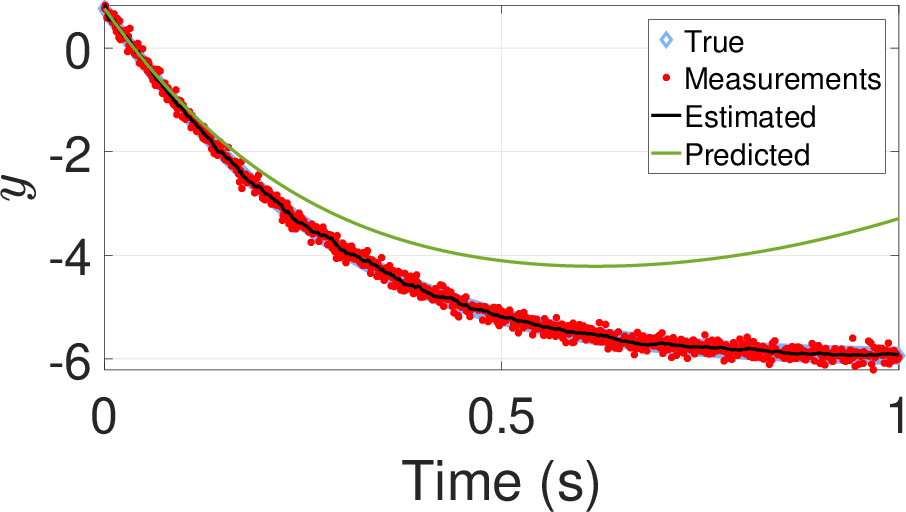}
    \par\smallskip\centering\footnotesize  \hesham{$y=E''_{q}V+\nu(t)$}
  \end{minipage}\hfill
  \begin{minipage}[t]{0.48\columnwidth}
    \centering
    \includegraphics[width=\linewidth]{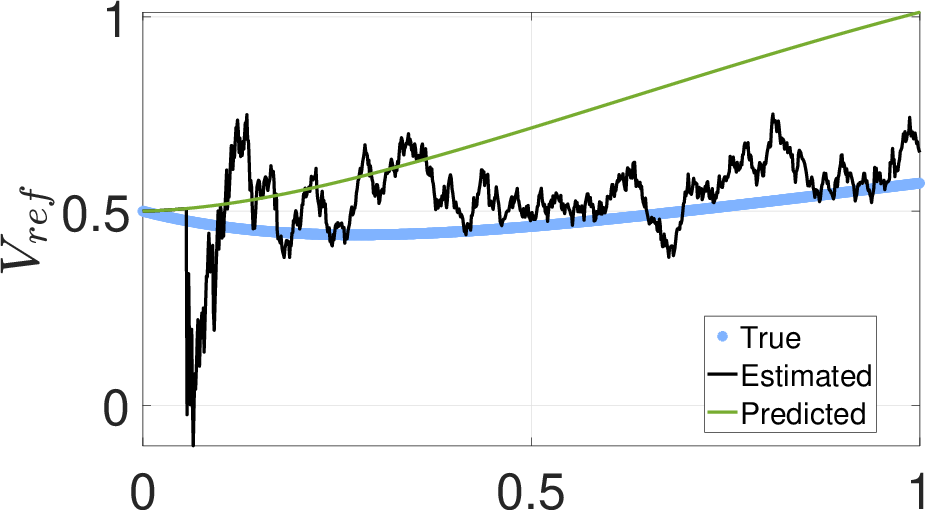}\\[1mm]
    \includegraphics[width=\linewidth]{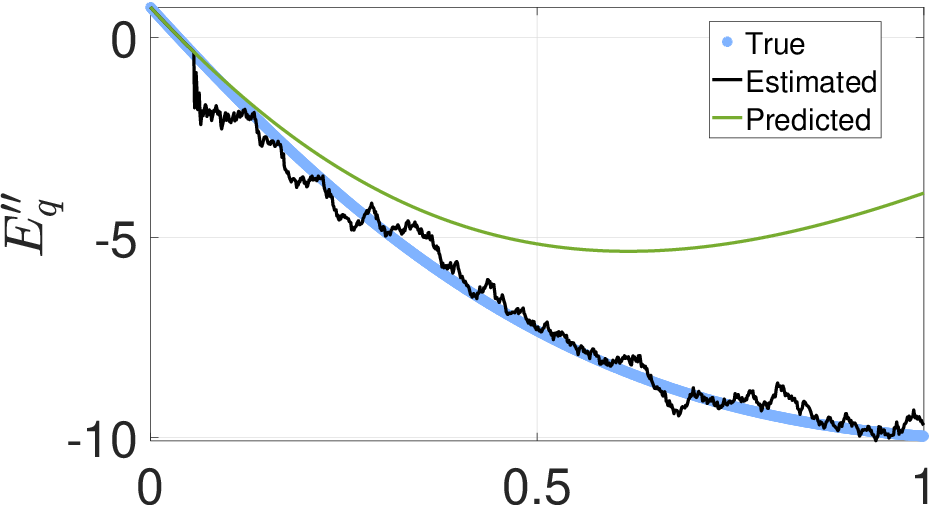}\\[1mm]
    \includegraphics[width=\linewidth]{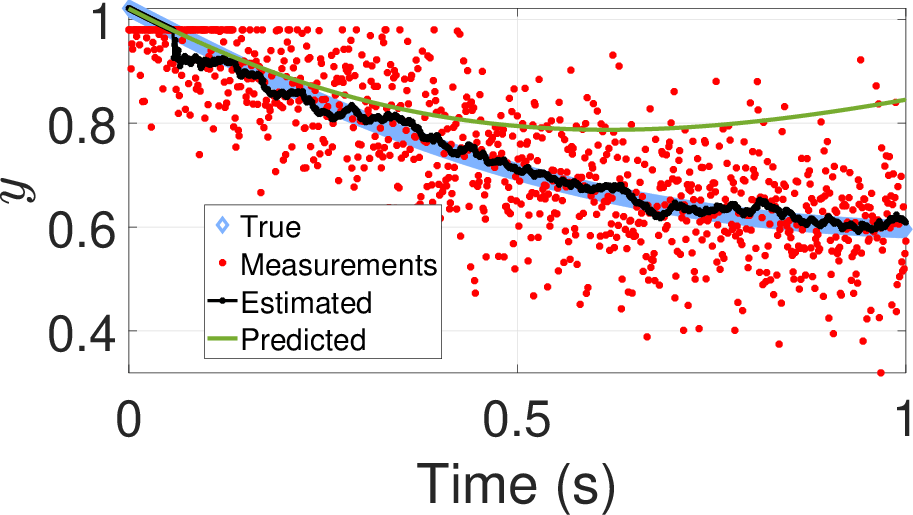}
    \par\smallskip\centering\footnotesize  \hesham{$y=\min(V+\nu(t),0.98)$}
  \end{minipage}
  \caption{\hesham{The left panel shows S-EKF with output $y=E''_{q}V+\nu(t)$ and the right panel shows S-EKF with output  $y=\textnormal{min}(V+\nu(t),0.98)$. The predicted (green) curve is the solution of  \eqref{ex:WTPS DAE} with $\Omega=0,\nu=0$, the true (light blue) curve is the solution of   \eqref{ex:WTPS DAE}, and the estimated (black) curve represents the filter states from the S-EKF algorithm.}}
  \label{fig:EKF_combined}
\end{figure}
\section{Conclusion}
We provided a novel sensitivity-based method for assessing local observability of smooth and nonsmooth DAE systems, which can also judge observable/non-observable states. 
In addition, we also introduced a novel sensitivity-based EKF algorithm (S-EKF), which uses the L-SERC observability test and sensitivities of DAE systems (smooth or nonsmooth) to provide accurate estimations of the system's states from the system's output. We demonstrated the new methods in a wind turbine power system application. Lastly, we note that the methods here are amenable to smooth/nonsmooth DAE systems that are non-autonomous and/or have explicit parametric dependence, the latter of which can be dealt with by amending the states with additional variables $x_{\theta_1},\ldots,x_{\theta_{n_p}}$ and augmenting the system with $n_p$ additional ODEs  $\dot{x}_{\theta_i}=0$.

\bibliography{references} 

\begin{thebibliography}{10}
\providecommand{\url}[1]{#1}
\csname url@rmstyle\endcsname
\providecommand{\newblock}{\relax}
\providecommand{\bibinfo}[2]{#2}
\providecommand\BIBentrySTDinterwordspacing{\spaceskip=0pt\relax}
\providecommand\BIBentryALTinterwordstretchfactor{4}
\providecommand\BIBentryALTinterwordspacing{\spaceskip=\fontdimen2\font plus
\BIBentryALTinterwordstretchfactor\fontdimen3\font minus \fontdimen4\font\relax}
\providecommand\BIBforeignlanguage[2]{{%
\expandafter\ifx\csname l@#1\endcsname\relax
\typeout{** WARNING: IEEEtran.bst: No hyphenation pattern has been}%
\typeout{** loaded for the language `#1'. Using the pattern for}%
\typeout{** the default language instead.}%
\else
\language=\csname l@#1\endcsname
\fi
#2}}

\bibitem{hermann1977nonlinear}
R.~Hermann and A.~Krener, ``Nonlinear controllability and observability,'' \emph{IEEE Transactions on Automatic Control}, vol.~22, no.~5, pp. 728--740, 1977.

\bibitem{stigter2018efficient}
J.~D. Stigter, L.~G. van Willigenburg, and J.~Molenaar, ``{An efficient method to assess local controllability and observability for non-linear systems},'' \emph{IFAC-PapersOnLine}, vol.~51, no.~2, pp. 535--540, 2018.

\bibitem{van2022sensitivity}
L.~G. van Willigenburg, J.~D. Stigter, and J.~Molenaar, ``Sensitivity matrices as keys to local structural system properties of large-scale nonlinear systems,'' \emph{Nonlinear Dynamics}, vol. 107, no.~3, pp. 2599--2618, 2022.

\bibitem{beard2012small}
R.~W. Beard and T.~W. McLain, \emph{Small Unmanned Aircraft: Theory and Practice}.\hskip 1em plus 0.5em minus 0.4em\relax Princeton University Press, 2012.

\bibitem{ljung1979asymptotic}
L.~Ljung, ``{Asymptotic behavior of the extended Kalman filter as a parameter estimator for linear systems},'' \emph{IEEE Transactions on Automatic Control}, vol.~24, no.~1, pp. 36--50, 1979.

\bibitem{pokhrel2023gradient}
S.~Pokhrel and S.~A. Eisa, ``{Gradient and Lie Bracket estimation of Extremum Seeking Systems: A novel geometric-based Kalman Filter and relaxed time-dependent stability condition},'' \emph{International Journal of Control, Automation and Systems}, vol.~21, no.~12, pp. 3839--3849, 2023.

\bibitem{stechlinski2025identifiability}
P.~Stechlinski, S.~A. Eisa, and H.~Abdelfattah, ``Identifiability and observability of nonsmooth systems via {Taylor}-like approximations,'' \emph{IEEE Transactions on Automatic Control}, vol.~70, no.~6, pp. 4178--4185, 2025.

\bibitem{chatzis2017discontinuous}
M.~N. Chatzis, E.~N. Chatzi, and S.~P. Triantafyllou, ``{A Discontinuous Extended Kalman Filter for non-smooth dynamic problems},'' \emph{Mechanical Systems and Signal Processing}, vol.~92, pp. 13--29, 2017.

\bibitem{brenan1995numerical}
K.~E. Brenan, S.~L. Campbell, and L.~R. Petzold, \emph{Numerical solution of initial-value problems in differential-algebraic equations}.\hskip 1em plus 0.5em minus 0.4em\relax SIAM, 1995.

\bibitem{kunkel2006differential}
P.~Kunkel, \emph{Differential-algebraic Equations: Analysis and Numerical Solution}.\hskip 1em plus 0.5em minus 0.4em\relax European Mathematical Society, 2006, vol.~2.

\bibitem{2015_Stechlinski_Barton_continuousdependence}
P.~Stechlinski and P.~I. Barton, ``Dependence of solutions of nonsmooth differential--algebraic equations on parameters,'' \emph{Journal of Differential Equations}, vol. 262, no.~3, pp. 2254--2285, 2017.

\bibitem{2016_Stechlinski_Barton_LD_DAEs}
------, ``{Generalized Derivatives of Differential–Algebraic Equations},'' \emph{Journal of Optimization Theory and Applications}, vol. 171, no.~1, pp. 1--26, 2016.

\bibitem{Stechlinski_Patrascu_Barton_DAEF}
P.~Stechlinski, M.~Patrascu, and P.~I. Barton, ``{Nonsmooth DAEs with applications in modeling phase changes},'' in \emph{Applications of Differential-Algebraic Equations: Examples and Benchmarks}.\hskip 1em plus 0.5em minus 0.4em\relax Springer, 2018, pp. 243--275.

\bibitem{2017_NonsmoothDAE_CACE}
------, ``Nonsmooth differential-algebraic equations in chemical engineering,'' \emph{Computers {\&} Chemical Engineering}, vol. 114, pp. 52--68, 2018.

\bibitem{pang2007strongly}
J.-S. Pang and J.~Shen, ``Strongly regular differential variational systems,'' \emph{IEEE Transactions on Automatic Control}, vol.~52, no.~2, pp. 242--255, 2007.

\bibitem{eisa2021sensitivity}
S.~A. Eisa and P.~Stechlinski, ``Sensitivity analysis of nonsmooth power control systems with an example of wind turbines,'' \emph{Communications in Nonlinear Science and Numerical Simulation}, vol.~95, p. 105633, 2021.

\bibitem{abdelfattah2025new}
H.~Abdelfattah, S.~A. Eisa, and P.~Stechlinski, ``A new nonsmooth optimal control framework for wind turbine power systems,'' \emph{Journal of the Franklin Institute}, vol. 362, no.~3, p. 107498, 2025.

\bibitem{campbell1991duality}
S.~L. Campbell, N.~K. Nichols, and W.~J. Terrell, ``Duality, observability, and controllability for linear time-varying descriptor systems,'' \emph{Circuits, Systems and Signal Processing}, vol.~10, no.~4, pp. 455--470, 1991.

\bibitem{campbell1991observability}
S.~L. Campbell and W.~J. Terrell, ``Observability of linear time-varying descriptor systems,'' \emph{SIAM Journal on Matrix Analysis and Applications}, vol.~12, no.~3, pp. 484--496, 1991.

\bibitem{hou2002causal}
M.~Hou and P.~Muller, ``Causal observability of descriptor systems,'' \emph{IEEE Transactions on Automatic Control}, vol.~44, no.~1, pp. 158--163, 2002.

\bibitem{terrell1997observability}
W.~J. Terrell, ``Observability of nonlinear differential algebraic systems,'' \emph{Circuits, Systems and Signal Processing}, vol.~16, no.~2, pp. 271--285, 1997.

\bibitem{montanari2024identifiability}
A.~N. Montanari, F.~Lamoline, R.~Bereza, and J.~Gon{\c{c}}alves, ``{Identifiability of differential-algebraic systems},'' \emph{arXiv preprint arXiv:2405.13818}, 2024.

\bibitem{becerra2001applying}
V.~M. Becerra, P.~Roberts, and G.~Griffiths, ``Applying the extended {Kalman} filter to systems described by nonlinear differential-algebraic equations,'' \emph{Control Engineering Practice}, vol.~9, no.~3, pp. 267--281, 2001.

\bibitem{mandela2010recursive}
R.~K. Mandela, R.~Rengaswamy, S.~Narasimhan, and L.~N. Sridhar, ``Recursive state estimation techniques for nonlinear differential algebraic systems,'' \emph{Chemical Engineering Science}, vol.~65, no.~16, pp. 4548--4556, 2010.

\bibitem{Khan2015_automaticdiff}
K.~A. Khan and P.~I. Barton, ``A vector forward mode of automatic differentiation for generalized derivative evaluation,'' \emph{Optimization Methods and Software}, vol.~30, no.~6, pp. 1185--1212, 2015.

\bibitem{Nesterov2005}
Y.~Nesterov, ``{Lexicographic differentiation of nonsmooth functions},'' \emph{Mathematical Programming}, vol. 104, pp. 669--700, 2005.

\bibitem{ref:OMS_generalizedderivatives}
P.~I. Barton, K.~A. Khan, P.~Stechlinski, and H.~A.~J. Watson, ``Computationally relevant generalized derivatives: theory, evaluation and applications,'' \emph{Optimization Methods and Software}, vol.~33, pp. 1030--1072, 2018.

\bibitem{Scholtes}
S.~Scholtes, \emph{Introduction to Piecewise Differentiable Equations}.\hskip 1em plus 0.5em minus 0.4em\relax Springer, 2012.

\bibitem{bagirov2020numerical_BGK+20}
A.~M. Bagirov, M.~Gaudioso, N.~Karmitsa, M.~M. M{\"a}kel{\"a}, and S.~Taheri, Eds., \emph{Numerical Nonsmooth Optimization: State of the Art Algorithms}.\hskip 1em plus 0.5em minus 0.4em\relax Springer, Cham, 2020.

\bibitem{Clarkebook}
F.~H. Clarke, \emph{Optimization and Nonsmooth Analysis}.\hskip 1em plus 0.5em minus 0.4em\relax SIAM, 1990.

\bibitem{abdelfattah2024parameter}
H.~Abdelfattah, P.~Stechlinski, and S.~A. Eisa, ``{Parameter identifiability and reduction for smooth and nonsmooth differential algebraic equation systems},'' \emph{IEEE Control Systems Letters}, 2024.

\bibitem{githubdae_obs}
H.~Abdelfattah, S.~Eisa, and P.~Stechlinski, ``{DAE L-SERC-Observability-Algorithm},'' \url{https://github.com/MDCL-UC/DAE_LSERC_Observability/}, 2025, {GitHub}.

\bibitem{brown1992introduction}
R.~Brown and P.~Hwang, \emph{Introduction to Random Signals and Applied {Kalman} Filtering}.\hskip 1em plus 0.5em minus 0.4em\relax Wiley, 1992, no. v. 2.

\bibitem{tsourakis2009effect}
G.~Tsourakis, B.~Nomikos, and C.~Vournas, ``Effect of wind parks with doubly fed asynchronous generators on small-signal stability,'' \emph{Electric Power Systems Research}, vol.~79, no.~1, pp. 190--200, 2009.

\bibitem{eisa2019modeling}
S.~A. Eisa, ``Modeling dynamics and control of type-3 {DFIG} wind turbines: Stability, {Q} droop function, control limits and extreme scenarios simulation,'' \emph{Electric Power Systems Research}, vol. 166, pp. 29--42, 2019.

\end{thebibliography}
\bibliographystyle{IEEEtran}

\end{document}